\documentclass[11pt]{amsart}

\usepackage{amsmath}
\usepackage{amsthm}
\usepackage{amssymb,graphics,amscd,amsfonts}
\usepackage{color}
\usepackage{comment}

\theoremstyle{plain}
\newtheorem{theorem}{Theorem}[section]
\newtheorem{proposition}[theorem]{Proposition}
\newtheorem{corollary}[theorem]{Corollary}
\newtheorem{lemma}[theorem]{Lemma}

\newtheorem{example}[theorem]{Example}
\newtheorem{definition}[theorem]{Definition}
\newtheorem{remark}[theorem]{Remark}

\theoremstyle{definition}
\newtheorem*{acknowledgements}{Acknowledgements}

\def\labelenumi{\rm (\theenumi)}
\newcommand{\tr}{\operatorname{tr}}
\newcommand{\id}{\operatorname{id}}

\title{Stability of anti-canonically balanced metrics}
\author{Shunsuke Saito and Ryosuke Takahashi}
\address{Graduate School of Mathematical Sciences, The University of Tokyo, 3-8-1 Komaba Meguro-ku Tokyo 153-8914, Japan}
\email{ssaito@ms.u-tokyo.ac.jp}
\address{Mathematical Institute, Tohoku University, 6-3, Aoba, Aramaki, Aoba-ku, Sendai, 980-8578, Japan}
\email{ryosuke.takahashi.a7@tohoku.ac.jp}
\begin{document}
\begin{abstract}
We study the asymptotic behavior of quantized Ding functionals along Bergman geodesic rays 
and prove that the slope at infinity can be expressed in terms of Donaldson-Futaki invariants and Chow weights. 
Based on the slope formula, we introduce a new algebro-geometric stability on Fano manifolds and show that the existence of anti-canonically balanced metrics implies our stability. 
The relation between our stability and others is also discussed.
As another application of the slope formula, we get the lower bound estimate on the Calabi like functionals on Fano manifolds. 
\end{abstract}
\maketitle
\tableofcontents
\section{Introduction}
In this paper, we study anti-canonically balanced metrics on Fano manifolds, introduced by Donaldson in \cite[Section 2.2.2]{D09} as a finite dimensional analogue of K\"ahler-Einstein metrics. 

Let $X$ be an $n$-dimensional Fano manifold and fix $k\ge 1$ so that $-kK_X$ is very ample. 
Let $\mathcal{H}(X, -K_X)$ be the space of smooth fiber metrics $\phi$ on $-K_X$ with the curvature $\omega_\phi:=(\sqrt{-1}/2\pi)\partial\overline{\partial}\phi$ positive
and $\mathcal{B}_k$ the space of Hermitian metrics on the finite dimensional vector space $H^0(X, -kK_X)$, which is a finite dimensional symmetric space of non-compact type. 

Following Donaldson \cite{D09}, we define the quantization map $Hilb_{k, \nu}\colon\mathcal{H}(X, -K_X)\to\mathcal{B}_k$ with respect to a volume form $\nu$ (with unit volume) to be
\begin{align*}
\langle\,\cdot\,, \,\cdot\,\rangle_{Hilb_{k, \nu}(\phi)}:=\int_X\langle\,\cdot\,, \,\cdot\,\rangle_{k\phi}\, d\nu
\end{align*}
and the dequantization map $FS_k\colon\mathcal{B}_k\to\mathcal{H}(X, -K_X)$ by
\begin{align*}
FS_k(H):=\frac{1}{k}\log\left(\frac{1}{N_k}\sum_{\alpha=1}^{N_k}|s_\alpha|^2\right), 
\end{align*}
where $N_k$ is the dimension of $H^0(X, -kK_X)$ and $(s_\alpha)$ is an $H$-orthonormal basis for $H^0(X, -kK_X)$. A Hermitian metric $H\in\mathcal{B}_k$ is called $k$-balanced metric with respect to $\nu$ if $H$ satisfies
\begin{align*}
(Hilb_{k, \nu}\circ FS_k)(H)=H. 
\end{align*}

The most well-understood balanced metrics are those with respect to the (normalized) Monge-Amp\`ere measure
\begin{align*}
MA(\phi):=\frac{\omega_\phi^n}{(-K_X)^n}, 
\end{align*}
where $(-K_X)^n$ is the top intersection number of $-K_X$. These metrics are called $k$-balanced metrics. The important fact is that  the existence of $k$-balanced metrics is equivalent to the Chow polystability of $(X, -K_X)$ at level $k$ (See \cite[Theorem 4]{PS03}). 
Note that these balanced metrics can be defined on general polarized manifolds and the theorem also holds. 

On a Fano manifold, a Hermitian metric $\phi\in\mathcal{H}(X, -K_X)$ defines another volume form $e^{-\phi}$ under the identification of fiber metrics on $-K_X$ with volume forms on $X$. Normalize $e^{-\phi}$ to be 
\begin{align*}
\mu_\phi:=\frac{e^{-\phi}}{\displaystyle\int_X e^{-\phi}}
\end{align*}
and simply write $Hilb_k(\phi):=Hilb_{k, \mu_\phi}(\phi)$. 
As introduced in \cite{D09}, a balanced metric defined by $Hilb_k$, that is, a Hermitian metric $H\in\mathcal{B}_k$ satisfying
\begin{align*}
(Hilb_k\circ FS_k)(H)=H
\end{align*}
is called an \textit{anti-canonically $k$-balanced metric}. We stress the point that anti-canonically balanced metrics make sense only on Fano manifolds as the name suggests. 

The first study on anti-canonically balanced metrics is given by Berman-Boucksom-Guedj-Zeriahi \cite{BBGZ13}. They characterized anti-canonically balanced metircs as the critical points of the quantized Ding functional and show that on K\"ahler-Einstein manifolds with a discrete automorphism group, the existence of the anti-canonically $k$-balanced metric for sufficiently large $k$ and the convergence of this sequence to the K\"ahler-Einstein metric at least $L^1$-topology. Later, Berman-Witt Nystr\"om \cite{BW14} treated the case of a continuous automorphism group and proved the same conclusion under the vanishing of all the higher Futaki invariants. 

We want to relate anti-canonically balanced metrics and algebro-geometric stability as in the case of balanced metrics. To do so, we study the slope at infinity of the quantized Ding functional along geodesic rays on $\mathcal{B}_k$. 
\begin{theorem}\label{slope}
Let $X$ be a Fano manifold, $(\mathcal{X}, \mathcal{L})$ a normal test configuration for $(X, -K_X)$ of exponent $k$ and $H\in\mathcal{B}_k$ a Hermitian metric on $H^0(X, -kK_X)$. Denoting by $(H_t)_t$  the Bergman geodesic ray associated with $(\mathcal{X}, \mathcal{L})$ and $H$ as explained in Section \ref{tc}, 
we have
\begin{align*}
\lim_{t\to\infty}\frac{d}{dt}D^{(k)}(H_t)+q=\frac{Fut_k(\mathcal{X}, \mathcal{L})}{kN_k}, 
\end{align*}
where $q$ is a non-negative rational number determined by the central fiber. The quantity $q$ vanishes if and only if $\mathcal{X}$ is $\mathbb{Q}$-Gorenstein with $\mathcal{L}$ isomorphic to $-kK_{\mathcal{X}/\mathbb{C}}$, and $\mathcal{X}_0$ is reduced, and its normalization has at worst log terminal singularities. 
\end{theorem}
The quantity in the right hand side is defined to be the sum of the Donaldson-Futaki invariant and the Chow weight of $(\mathcal{X}, \mathcal{L})$, and $Fut_k(\mathcal{X}, \mathcal{L})$ is called the \textit{quantized Futaki invariant}. Then, we introduce a new stability on a Fano manifold $X$, \textit{F-stability}, using the quantized Futaki invariant and show the following:

\begin{theorem}\label{balanced}
Let $X$ be a Fano manifold admitting an anti-canonically $k$-balanced metric. Then, $X$ is F-polystable at level $k$. 
\end{theorem}

We next compare Chow stability and our F-stability. 

\begin{theorem}\label{Chow->F}
Asymptotic Chow polystability $($resp. stability or semistability$)$ implies asymptotic F-polystability $($resp. stability or semistability$)$. 
\end{theorem}
We also discuss the relations between asymptotic F-stability with uniform K-stability and K-semistability in Section \ref{vsF}. 

As an application of Theorem \ref{slope}, we prove the lower bound estimate on the $L^q$-norm of the function
\begin{align*}
B(\phi):=\frac{n!\mu_\phi}{\omega_\phi^{n}}-\frac{n!}{(-K_X)^n}, \quad\phi\in\mathcal{H}(X, -K_X). 
\end{align*}
Note that $\phi$ is a K\"ahler-Einstein metric if and only if $B(\phi)=0$. In other words, $B(\phi)$ measures the deviation from $\phi$ being a K\"ahler-Einstein metric. 

\begin{theorem}\label{LB}
Let $p$ be a positive even integer and $q$ the H\"older conjugate of $p$. 
Given a Hermitian metirc $\phi\in\mathcal{H}(X, -K_X)$ and a normal test configuration $(\mathcal{X}, \mathcal{L})$ for $(X, -K_X)$ with non-zero $p$-norm, we have
\begin{align*}
||B(\phi)||_{L^q(\omega_\phi^n/n!)}\ge -\frac{DF(\mathcal{X}, \mathcal{L})}{||(\mathcal{X}, \mathcal{L})||_p}, 
\end{align*}
where $||\cdot||_{L^q(\omega_\phi^n/n!)}$ denotes the $L^q$-norm with respect to $\omega^n_\phi/n!$. 
\end{theorem}
This is an analogue of the Donaldson's result \cite[Theorem 2]{D05} in Fano case.
Although this result was already proved by Hisamoto \cite[Theorem 1.3]{H16} for any $p\in[1, \infty]$ (see also \cite[Theorem 4.3]{B16}), the viewpoints are different. We will prove it via a finite-dimensional argument following Donaldson, while he took an energy theoretic approach. 

\begin{acknowledgements}
The first author is grateful to Dr. Yoshinori Hashimoto and Professor Yasufumi Nitta for stimulating discussion. 
The second author would like to thank Professor Shigetoshi Bando and Professor Ryoichi Kobayashi for useful conversations on this article. 
The first author is supported by JSPS KAKENHI Grant Number
15J06855 and the Program for Leading Graduate
Schools, MEXT, Japan. The second author is supported by Grant-in-Aid for JSPS Fellows Number 25-3077 and 16J01211.
\end{acknowledgements}
\section{Preliminaries}
\subsection{Test configurations and Bergman geodesic rays}\label{tc}
We assume in this section that $(X, L)$ be a polarized manifold. 
\begin{definition}
A test configuration for $(X, L)$ of exponent $k$ consists of the following data
\begin{enumerate}
\item a scheme $\mathcal{X}$ with a $\mathbb{C}^\ast$-action $\rho$$;$
\item a $\mathbb{C}^\ast$-equivariant flat and proper morphism $\pi\colon\mathcal{X}\to\mathbb{C}$, where $\mathbb{C}^\ast$ acts on $\mathbb{C}$ by the standard multiplication$;$
\item a $\mathbb{C}^\ast$-linearized $\pi$-very ample line bundle $\mathcal{L}$ on $\mathcal{X}$$;$
\item an isomorphism $(\mathcal{X}_1, \mathcal{L}_1)\cong (X, kL)$. 
\end{enumerate}
A test configuration $(\mathcal{X}, \mathcal{L})$ is called a product configuration if $\mathcal{X}\cong X\times \mathbb{C}$, and a trivial configuration if in addition $\mathbb{C}^\ast$ acts only on the second factor. A test configuration $(\mathcal{X}, \mathcal{L})$ is called normal if $\mathcal{X}$ is a normal variety. 
For a Fano manifold $(X, -K_X)$ with the anti-canonical polarization, a test configuration $(\mathcal{X}, \mathcal{L})$ is said to be special if the central fiber $\mathcal{X}_0$ is a normal variety with at worst log terminal singularities. 
\end{definition}

Fix $k\ge 1$ so that $kL$ is very ample. The following proposition relates test configurations with fixed exponent to finite dimensional objects. 
\begin{proposition}[{\cite[Proposition 3.7]{RT07}}]\label{1-PS}
A one-parameter subgroup of $GL(H^0(X, kL))$ is equivalent to the data of a test configuration for $(X, L)$ of exponent $k$. 
\end{proposition}
\begin{proof}
Let $\sigma\colon \mathbb{C}^\ast\to GL(H^0(X, kL))$ be a one-parameter subgroup and $\Phi_{|kL|}\colon X\hookrightarrow \mathbb{P} H^0(X, kL)^\ast$ the closed embedding defined by $|kL|$. 
We define $\mathcal{X}$ by the Zariski closure of the image under the embedding $X\times\mathbb{C}^\ast\hookrightarrow \mathbb{P} H^0(X, kL)^\ast\times \mathbb{C}$ defined by $(x, \tau)\mapsto (\sigma^\ast(\tau)\Phi_{|kL|}(x), \tau)$, that is, $\mathcal{X}_0$ is defined as the flat limit of the image of $X$ under $\sigma^\ast$ as $\tau\to 0$, 
and put $\mathcal{L}:=\mathcal{O}_{\mathcal{X}}(1)$. This gives a test configuration for $(X,L)$ of exponent $k$. The converse direction is spelled out below. 
\end{proof}

Let $(\mathcal{X}, \mathcal{L})$ be a test configuration for $(X, L)$ of exponent $k$. The $\mathbb{C}^\ast$-action $\rho$ on $(\mathcal{X}, \mathcal{L})$ induces an isomorphism
$
\rho(\tau, w)\colon H^0(\mathcal{X}_w, \mathcal{L}_w)\to H^0(\mathcal{X}_{\tau w}, \mathcal{L}_{\tau w})$
for any $\tau\in\mathbb{C}^\ast$ and $w\in\mathbb{C}$. Put
$\rho(\tau):=\rho(\tau, 1)\colon H^0(X, kL)\to H^0(\mathcal{X}_\tau, \mathcal{L}_\tau)$ and 
$\rho_0(\tau):=\rho(\tau, 0)\colon H^0(\mathcal{X}_0, \mathcal{L}_0)\to H^0(\mathcal{X}_0, \mathcal{L}_0)$. 
Let $A_k$ denote the infinitesimal generator of $\rho_0$. Fix a Hermitian metric $H\in\mathcal{B}_k$ on $H^0(X, kL)$. 

\begin{theorem}[{\cite[Lemma 2]{D05}, \cite[Lemma 2.1]{PS07}}]\label{equivtriv}
There exists an isomorphism $\Theta_k\colon H^0(\mathcal{X}_0, \mathcal{L}_0)\to H^0(X, kL)$ satisfying
\begin{enumerate}
\item $\Theta_k$ is derived from a $\mathbb{C}^\ast$-equivariant embedding
\begin{align*}
(\mathcal{X}, \mathcal{L})\hookrightarrow (\mathbb{P} H^0(\mathcal{X}_0, \mathcal{L}_0)\times\mathbb{C}, \mathcal{O}(1))
\end{align*}
whose restriction on the central fiber gives the closed embedding defined by $|\mathcal{L}_0|$$;$
\item\label{Hermite} $A_k$ is Hermitian with respect to $H_k:=\Theta_k^\ast H$. 
\end{enumerate}
The Hermitian metric $H_k$ is independent of $\Theta_k$. Moreover, such a $\Theta_k$ is unique up to an isometry on $(H^0(\mathcal{X}_0, \mathcal{L}_0), H_k)$ commuting with $\rho_0$. 
$\Theta_k$ is called a regular Hermitian generator. 
\end{theorem}

Using a regular Hermitian generator $\Theta_k$, we can define a one-parameter subgroup $\lambda\colon \mathbb{C}^\ast\to GL(H^0(X, kL))$, so that $\Theta_k$ is a $\mathbb{C}^\ast$-equivariant isomorphism. More concretely, for $\tau\in\mathbb{C}^\ast$, we define
\begin{align*}
\lambda(\tau):=\Theta_k\circ \rho_0(\tau)\circ \Theta_k^{-1}. 
\end{align*}
Note that $\lambda$ is independent of the choice of a regular Hermitian generator $\Theta_k$. Indeed, for another regular Hermitian generator $\Theta_k^\prime$, there exists a unitary endomorphism $U_k$ commuting with $\rho_0$ such that $\Theta_k=\Theta_k^\prime\circ U_k$. Then, 
\begin{align*}
\Theta_k\circ\rho_0(\tau)\circ\Theta_k^{-1}
&=(\Theta_k^\prime\circ U_k)\circ\rho_0(\tau)\circ(U_k^{-1}\circ\Theta_k^\prime{}^{-1})\\
&=\Theta_k^\prime\circ\rho_0(\tau)\circ\Theta_k^{\prime}{}^{-1}. 
\end{align*}
This $\lambda$ is the desired one-parameter subgroup corresponding to $(\mathcal{X}, \mathcal{L})$. 

Next, we show the way to associate a Bergman geodesic ray (i.e., a geodesic ray on $\mathcal{B}_k$) with 
a test configuration $(\mathcal{X}, \mathcal{L})$ of exponent $k$. Let $H$, $\Theta_k$ be as above. 
For $\tau\in\mathbb{C}^\ast$, we define a Hermitian metirc $H_\tau$ on $H^0(\mathcal{X}_\tau, \mathcal{L}_\tau)$ by
\begin{align*}
H_\tau:=((\rho(\tau)\circ \Theta_k\circ\rho_0(\tau^{-1}))^{-1})^\ast H_k
=(\rho(\tau)^{-1})^\ast\lambda(\tau)^\ast H. 
\end{align*}
Since $\lambda$ is independent of $\Theta_k$, so is $H_\tau$. 
Note that according to Theorem \ref{equivtriv} (\ref{Hermite}), $H$ is $S^1$-invariant. 
Then, we can use the real logarithmic coordinate $t=-\log|\tau|^2$ on the punctured unit disc $\Delta^\ast\subset\mathbb{C}$ centered at the origin. 
By means of the isomorphism $\rho(\tau)$, we get a geodesic
\begin{align*}
H_t:=\rho(\tau)^\ast H_\tau=\lambda(e^{-\frac{1}{2}t})^\ast H=e^{-tA}H
\end{align*}
on $\mathcal{B}_k$ parametrized by $t\in[0, \infty)$, where $A$ denotes the infinitesimal generator of $\lambda$. 
We will call it the Bergman geodesic ray associated with $(\mathcal{X} ,\mathcal{L})$ and $H$. 

Finally, we prove the following lemma for later use, which says that the Bergman type metric defined by $(H_\tau)_\tau$ extends to $\tau=0$. 
Put $(\mathcal{X}_\Delta, \mathcal{L}_\Delta):=(\pi^{-1}(\Delta), \mathcal{L}|_{\pi^{-1}(\Delta)})$, and
 $(\mathcal{X}_{\Delta^\ast}, \mathcal{L}_{\Delta^\ast}):=(\pi^{-1}(\Delta^\ast), \mathcal{L}|_{\pi^{-1}(\Delta^\ast)})$. 
\begin{lemma}\label{extension}
Let $(\mathcal{X}, \mathcal{L})$, $H$, and $(H_\tau)_\tau$ be as above. A locally bounded metric $\phi$ on $\mathcal{L}_{\Delta^\ast}$ defined by
\begin{align*}
\phi_\tau:=k FS_k(H_\tau), \quad \tau \in\Delta^\ast
\end{align*}
gives an $S^1$-invariant locally bounded metric on $\mathcal{L}_\Delta$ with positive curvature current. 
\end{lemma}
\begin{proof}
Fix an $H$-orthonormal basis $(s_\alpha)$ for $H^0(X, -kK_X)$ consisting of the weight vectors of $\lambda$: 
\begin{align*}
\lambda(\tau)s_\alpha=\tau^{m_\alpha}s_\alpha, 
\end{align*}
where $m_\alpha$ is a weight of $\lambda$. Then, $(\tau^{-m_\alpha}\rho(\tau)s_\alpha)$ is an $H_\tau$-orthonomal basis. Hence, 
\begin{align*}
\phi_\tau=\log\left(
\frac{1}{N_k}\sum_{\alpha=1}^{N_k}|\tau^{-m_\alpha}\rho(\tau)s_\alpha|^2
\right). 
\end{align*}
It suffices to show that each $\tau^{-m_\alpha}\rho(\tau)s_\alpha$ extends holomorphically to $\tau=0$. We follow the argument of \cite[Lemma 6.1]{WN12}. 
To begin with, we prepare the notations. Define a holomorphic section $\overline{s_\alpha}\in H^0(\mathcal{X}\setminus\mathcal{X}_0, \mathcal{L})$ by
\begin{align*}
\overline{s_\alpha}(\rho(\tau)x):=\rho(\tau)s_\alpha(x), \quad \tau\in\mathbb{C}^\ast, x\in X. 
\end{align*}
Let $w$ denote the global coordinate on $\mathbb{C}$ and be identified with the projection $\mathcal{X}\to\mathbb{C}$. We also regard it as a section of the trivial line bundle over $\mathcal{X}$. Then $w^{-m_\alpha}\overline{s_\alpha}$ is a holomorphic section of $\mathcal{L}$ over $\mathcal{X}\setminus\mathcal{X}_0$ and $w^{-m_\alpha}\overline{s_\alpha}=w^{-m_\alpha}\rho(w)s_\alpha$ for any $w\in\mathbb{C}^\ast$. We now prove the claim. 
Since $\pi_\ast\mathcal{L}\to\mathbb{C}$ is $\mathbb{C}^\ast$-equivariantly trivial, there exist global sections $\sigma_1, \ldots, \sigma_{N_k}$ of $\pi_\ast\mathcal{L}$ such that
\begin{enumerate}
\item for each $w\in\mathbb{C}$, $(\sigma_1(w), \ldots, \sigma_{N_k}(w))$ is a basis for $H^0(\mathcal{X}_w, \mathcal{L}_w)$;  
\item there exists an invertible matrix $(f_{\alpha\beta}(\tau))$ with coefficients in $\mathbb{C}[\tau, \tau^{-1}]$ satisfying
\begin{align}\label{sigma}
\rho(\tau)\sigma_\alpha=\sum_\beta f_{\alpha\beta}(\tau)\sigma_\beta, 
\end{align}
for any $\tau\in\mathbb{C}^\ast$. 
\end{enumerate}
We may assume that $\sigma_\alpha(1)=s_\alpha$ for $\alpha=1, \ldots, N_k$. Then, 
\begin{align*}
\rho_0(\tau)\sigma_\alpha(0)=\tau^{m_\alpha}\sigma_\alpha(0). 
\end{align*}
 On the other hand, restricting (\ref{sigma}) to the central fiber gives
\begin{align*}
\rho_0(\tau)\sigma_\alpha(0)=\sum_\beta f_{\alpha\beta}(\tau)\sigma_\beta(0). 
\end{align*}
Combining them, we have $f_{\alpha\beta}(\tau)=\tau^{m_\alpha}\delta_{\alpha\beta}$. Hence, for any $w\in\mathbb{C}^\ast$, 
\begin{align*}
\overline{s_\alpha}(w)
=(\rho(w)s_\alpha)(\tau)
=(w^{m_\alpha}\sigma_\alpha)(w), 
\end{align*}
so that $w^{-m_\alpha}\overline{s_\alpha}=\sigma_\alpha$. $\sigma_\alpha$ is holomorphic on $\mathbb{C}$, so is $w^{-m_\alpha}\overline{s_\alpha}$ as desired. \qedhere
\end{proof}

\subsection{Chow weights and Donaldson-Futaki invariants}\label{Chow&DF}
In this section, we recall the definition of Chow weights and Donaldson-Futaki invariants. 

Let $(X, L)$ be an $n$-dimensional polarized manifold and $(\mathcal{X}, \mathcal{L})$ a test configuration for $(X, L)$ of exponent $k$. Denote by $N_{km}$ the dimension of $H^0(\mathcal{X}_0, m\mathcal{L}_0)$ and by $w_{km}$ the total weight of the $\mathbb{C}^\ast$-action on  $H^0(\mathcal{X}_0, m\mathcal{L}_0)$ induced by that on $(\mathcal{X}, \mathcal{L})$. For large $m$, we have expansions: 
\begin{align*}
N_{km}&=a_0(km)^n+a_1(km)^{n-1}+\cdots+a_n, \\
w_{km}&=b_0(km)^{n+1}+b_1(km)^n+\cdots+b_{n+1}. 
\end{align*}
The Chow weight of $(\mathcal{X}, \mathcal{L})$ is defined by
\begin{align*}
Chow_k(\mathcal{X}, \mathcal{L}):=\frac{b_0}{a_0}-\frac{w_k}{kN_k}. 
\end{align*}
The Donaldson-Futaki invariant of $(\mathcal{X}, \mathcal{L})$ is
\begin{align*}
DF(\mathcal{X}, \mathcal{L}):=2\frac{a_1b_0-a_0b_1}{a_0^2}. 
\end{align*}
Note that these invariants are independent of the choice of a $\mathbb{C}^\ast$-linearization of $\mathcal{L}$. We also note that the Donaldson-Futaki invariant is unchanged by replacing $\mathcal{L}$ with a tensor power, while the Chow weight is not. In fact, we have
\begin{align*}
Chow_{km}(\mathcal{X}, m\mathcal{L})=\frac{b_0}{a_0}-\frac{w_{km}}{kmN_{km}}, 
\end{align*}
from which one can easily get 
\begin{align}\label{Chow-DF}
\lim_{m\to\infty} km\,Chow_{km}(\mathcal{X}, m\mathcal{L})=\frac{1}{2}DF(\mathcal{X}, \mathcal{L}). 
\end{align}

With these invariants above, we can define Chow stability and K-semistability of polarized manifolds. 
\begin{definition}
A polarized manifold $(X, L)$ is said to be
\begin{enumerate}
\renewcommand{\labelenumi}{(\roman{enumi})}
\item
\begin{enumerate}
\item Chow semistable at level $k$ if $Chow_k(\mathcal{X}, \mathcal{L})\ge 0$ holds for any test configuration $(\mathcal{X}, \mathcal{L})$ for $(X, -K_X)$ of exponent $k$. 
\item Chow polystable at level $k$ if $(X, L)$ is Chow semistable at level $k$ and $Chow_k(\mathcal{X}, \mathcal{L})=0$ if and only if $(\mathcal{X}, \mathcal{L})$ is product. 
\item Chow stable at level $k$ if $(X, L)$ is Chow semistable at level $k$ and $Chow_k(\mathcal{X}, \mathcal{L})=0$ if and only if $(\mathcal{X}, \mathcal{L})$ is trivial.  
\item asymptotically Chow polystable $($resp. stable or semistable$)$ if there exists a $k_0>0$ such that $(X, L)$ is Chow polystable $($resp. stable or semistable$)$ at level $k$ for all $k\ge k_0$. 
\end{enumerate}
\item K-semistable if $DF(\mathcal{X}, \mathcal{L})\ge 0$ holds for any test configuration $(\mathcal{X}, \mathcal{L})$ for $(X, -K_X)$. 
\end{enumerate}
\end{definition}

Note that (\ref{Chow-DF}) gives the following relation between two semistabilities:

\begin{proposition}[{\cite[Theorem 3.9]{RT07}}]\label{Chow->K}
Asymptotic Chow semistability implies K-semistability. 
\end{proposition}

We briefly explain how higher Futaki invariants
 $\mathcal{F}_{\mathrm{Td}^{(1)}}, \ldots, \mathcal{F}_{\mathrm{Td}^{(n)}}$, obstructions to asymptotic Chow semistability, are related to Chow weights, following Della Vedova-Zuddas \cite[Proposition 2.2]{DVZ12}. 
Let $V$ be a holomorphic vector field on $X$ whose real part generates $S^1$. 
As explained in Proposition \ref{1-PS}, $V$ defines a product configuration $(\mathcal{X}, \mathcal{L})$ for $(X, L)$ of exponent $k$. 
Using the equivariant Riemann-Roch theorem, we get
\begin{align}\label{Chow&hFut}
Chow_{km}(\mathcal{X}, m\mathcal{L})
&=-\frac{b_{n+1}}{kmN_{km}}-\frac{a_0}{kmN_{km}}\sum_{p=1}^n
\frac{a_0b_p-a_pb_0}{a_0^2}\,(km)^{n+1-p}\\
&=-\frac{1}{kmN_{km}}\sum_{p=1}^n\frac{(km)^{n+1-p}}{(n+1-p)!}\mathcal{F}_{\mathrm{Td}^{(p)}}(V)\notag,  
\end{align}
for sufficiently large $m$. Note that the smoothness of $X$ implies $b_{n+1}=0$. 

We end this section by defining norms of test configurations for later use. Let $p\ge 1$. 
Given a test configuration as above, denote by $A_{km}$ the infinitesimal generator of the $\mathbb{C}^\ast$-action on $H^0(\mathcal{X}_0, m\mathcal{L}_0)$ and by $\underline{A}_{km}$ the trace-free part of $A_{km}$. 
We define the $p$-norm $||(\mathcal{X}, \mathcal{L})||_p$ to be the $p$-th root of the leading coefficient in
\begin{align*}
\tr(\underline{A}_{km}^p)
=||(\mathcal{X}, \mathcal{L})||_p^p (km)^{n+p}+O(m^{n+p-1})
\end{align*}
for large $m$. This is unchanged if we replace $\mathcal{L}$ by a power. 

\subsection{Kempf-Ness type functionals and its quantizations}
The aim of this section is to recall the definition of energy functionals. 
Let $X$ be an $n$-dimensional Fano manifold. Fix a Hermitian metric $\phi_0\in\mathcal{H}(X. -K_X)$ and put $\omega_0:=\omega_{\phi_0}$. 
For a smooth Hermitian metric $\phi\in\mathcal{H}(X, -K_X)$, 
we define the \emph{Monge-Amp\`ere energy} $\mathcal{E}$ and the \emph{Ding functional} $D$ by
\begin{align*}
\mathcal{E}(\phi)&:=\frac{1}{n+1}\sum_{i=0}^n\int_X(\phi-\phi_0)\, \omega_\phi^{n-i}\wedge\omega_0^i, \\
D(\phi)&:=-\frac{1}{(-K_X)^n}\mathcal{E}(\phi)+\mathcal{L}(\phi), \quad \mathcal{L}(\phi):=-\log\int_X e^{-\phi}. 
\end{align*}

For a Hermitian metric $H\in\mathcal{B}_k$, we also define the \emph{quantized Monge-Amp\`ere energy} $\mathcal{E}^{(k)}$, the \emph{balancing energy} $Z_k$, and the \emph{quantized Ding functional} $\mathcal{D}^{(k)}$ by
\begin{align*}
\mathcal{E}^{(k)}(H)&:=-\frac{1}{kN_k}\log\det H, \\
Z_k(H)&:=\frac{(-K_X)^n}{n!}k^{n+1}\left(\frac{1}{(-K_X)^n}\mathcal{E}(FS_k(H))-\mathcal{E}^{(k)}(H)
\right), \\
D^{(k)}&:=-\mathcal{E}^{(k)}(H)+\mathcal{L}(FS_k(H)), 
\end{align*}
where the determinant is taken with respect to $Hilb_k(\phi_0)$. 

We collect some properties of these functionals. 
\begin{proposition}\label{facts}
Let $H\in\mathcal{B}_k$ be a Hermitian metric on $H^0(X, -kK_X)$. 
\begin{enumerate}
\item\label{crit} $H$ is a critical point of $D^{(k)}$ if and only if $H$ is an anti-canonically $k$-balanced metric. 
\item \label{convexity} $D^{(k)}$ is convex along Bergman geodesic rays. 
\item\label{D^k} We have
\begin{align*}
D^{(k)}(H)=D(FS_k(H))+\frac{n!}{k^{n+1}(-K_X)^n} Z_k(H). 
\end{align*}
\item\label{product} Let $(\mathcal{X}, \mathcal{L})$ be a test configuration for $(X, -K_X)$ of exponent $k$ and $(H_t)_t$ the Bergman geodesic ray associated with $(\mathcal{X}, \mathcal{L})$ and $H$. 
If $D^{(k)}(H_t)$ is affine in $t$ on $[0, \infty)$, then $(\mathcal{X}, \mathcal{L})$ is a product configuration. 
\end{enumerate}
\end{proposition}
\begin{proof}
(\ref{crit}) and (\ref{convexity}) were proved in Lemma 7.4 and Lemma 6.5 of \cite{BBGZ13}, respectively. One could also prove them using Proposition \ref{derivative}. 
(\ref{D^k}) is trivial. 
We now start the proof of (\ref{product}). 
By combining our assumption on $D^{(k)}(H_t)$ with the convexity of $D\circ FS_k$ and $Z_k$, (\ref{D^k}) shows that $Z_k(H_t)$ is affine in $t$. To complete the proof, we need the explicit formula for the second derivative of $Z_k(H_t)$. 
Denote $H_t=e^{-tA}H$. Let $V_A$ be the holomorphic vector field on $\mathbb{P} H^0(X, -kK_X)^\ast$ defined by the Hermitian matrix $A$ and $V_A^\perp$ the normal part of $V_A$ with respect to the Fubini-Study metric induced by $H_t$.  
It was proved in \cite[Lemma 17]{F10} that
\begin{align*}
\frac{d^2}{dt^2}Z_k(H_t)=\frac{k^n}{n!}\int_X |V_A^\perp|^2_{k\omega_{FS_k(H_t)}}\omega^n_{FS_k(H_t)}. 
\end{align*}
This implies that $V_A$ is tangent to the image of $X$ under the closed embedding $X\hookrightarrow \mathbb{P} H^0(X, -kK_X)^\ast$, so that the central fiber $\mathcal{X}_0$ is isomorphic to $X$ by the proof of Proposition \ref{1-PS}. 
\end{proof}

\section{Quantized Futaki invariants and F-stability}
In this section we introduce quantized Futaki invariants and F-stability. 

Let $X$ be an $n$-dimensional Fano manifold. Fix $k\ge 1$ so that $-kK_X$ is very ample. 

\begin{definition}
Given a test configuration $(\mathcal{X}, \mathcal{L})$ for $(X, -K_X)$ of exponent $k$, the quantized Futaki invariant at level $k$ is defined to be
\begin{align*}
Fut_k(\mathcal{X}, \mathcal{L}):=kN_k(DF(\mathcal{X}, \mathcal{L})+Chow_k(\mathcal{X}, \mathcal{L})). 
\end{align*}
\end{definition}
We remark that this invariant is independent of the choice of a $\mathbb{C}^\ast$-linearization of $\mathcal{L}$, since so are the Donaldson-Futaki invariant  and the Chow weight. 

The following lemma explains why we call $Fut_k(\mathcal{X}, \mathcal{L})$ the quantized Futaki invariant. 
\begin{lemma}\label{BW}
If $(\mathcal{X}, \mathcal{L})$ is a special test configuration of exponent $k$, then $Fut_k(\mathcal{X}, \mathcal{L})$ coincides with the quantized Futaki invariant introduced by Berman-Witt Nystr\"om in \cite[Section 4.4]{BW14}. 
\end{lemma}
\begin{remark}
Before giving the proof, we should recall the definition of quantized Futaki invariants by Berman-Witt Nystr\"om. Let $(\mathcal{X}, \mathcal{L})$ be a special test configuration of exponent $k$. 
By \cite[Lemma 2.2]{B16}, $\mathcal{X}$ is a normal $\mathbb{Q}$-Gorenstein variety, and $\mathcal{L}$ is isomorphic to the relative pluri-anti-canonical divisor $-kK_{\mathcal{X}/\mathbb{C}}$. 
Then, we can lift the $\mathbb{C}^\ast$-action on $\mathcal{X}$ automatically to the tangent bundle of the regular part of $\mathcal{X}$, and eventually to $-kK_{\mathcal{X}/\mathbb{C}}$. 
This particular linearization of $\mathcal{L}=-kK_{\mathcal{X}/\mathbb{C}}$ is called ``canonical''. 
Note that this is not necessarily the same as the a priori linearization of $\mathcal{L}$. 
Given a special test configuration $(\mathcal{X}, \mathcal{L})$ with the canonical linearization, 
Berman-Witt Nystr\"om defined the quantized Futaki invariant at level $k$ to be the opposite sign of the 
total weight of the $\mathbb{C}^\ast$-action on $H^0(\mathcal{X}_0, -kK_{\mathcal{X}_0})$. 
Let us stress the point that they only considered special test configurations with the canonical linearization. Our definition is considered as a generalization of theirs. 
\end{remark}
\begin{proof}
We use the notation as used in Section \ref{Chow&DF}. Since our $Fut_k(\mathcal{X}, \mathcal{L})$ is independent of the linearization, we may choose the canonical linearization. Then, $-w_k$ is the quantized Futaki invariant defined by Berman-Witt Nystr\"om. 

The key point in the proof is the formula
\begin{align}\label{DF=F1}
DF(\mathcal{X}, \mathcal{L})=-\frac{b_0}{a_0}. 
\end{align}
Once we have established (\ref{DF=F1}), we have
\begin{align*}
Fut_k(\mathcal{X}, \mathcal{L})
&=kN_k(DF(\mathcal{X}, \mathcal{L})+Chow_k(\mathcal{X}, \mathcal{L}))\\
&=kN_k\left(-\frac{b_0}{a_0}+\frac{b_0}{a_0}+\frac{-w_k}{kN_k}\right)\\
&=-w_k. 
\end{align*}

To prove (\ref{DF=F1}), there are two ways. The first approach is to apply the equivariant Riemann-Roch formula to a normal variety $\mathcal{X}_0$. Consult for example \cite[Lemma 1.2]{WZZ16}. 
The second one is to consider the compactification $(\overline{\mathcal{X}}, \overline{\mathcal{L}})\to\mathbb{P}^1$ of $(\mathcal{X}, \mathcal{L})$ whose $\infty$-fiber has the trivial $\mathbb{C}^\ast$-action and apply the two-term asymptotic Riemann-Roch theorem to a normal variety $\overline{\mathcal{X}}$. See for example \cite[Proposition 3.12 (iv)]{BHJ15}. The latter approach gives
\begin{align*}
w_{km}
&=\frac{\overline{\mathcal{L}}^{n+1}}{k^{n+1}(n+1)!}(km)^{n+1}
+\frac{(-K_{\overline{\mathcal{X}}/\mathbb{P}^1}\cdot\overline{\mathcal{L}}^n)}{2k^nn!}(km)^n+O(m^{n-1})\\
&=\frac{(-K_{\overline{\mathcal{X}}/\mathbb{P}^1})^{n+1}}{(n+1)!}(km)^{n+1}
+\frac{(-K_{\overline{\mathcal{X}}/\mathbb{P}^1})^{n+1}}{2n!}(km)^n+O(m^{n-1})
\end{align*}
for large $m$. 
On the other hand, the two-term asymptotic Riemann-Roch theorem on $(X, -kK_X)$ yields
\begin{align*}
N_{km}=\frac{(-K_X)^n}{n!}(km)^n+\frac{(-K_X)^n}{2(n-1)!}(km)^{n-1}+O(m^{n-2}). 
\end{align*}
It follows that 
\begin{align*}
b_1=\frac{n+1}{2}b_0, \quad a_1=\frac{n}{2}a_0, 
\end{align*}
which proves (\ref{DF=F1}). 
\end{proof}
Finally, we introduce a new stability of Fano manifolds. 

\begin{definition}
A Fano manifold $X$ is said to be
\begin{enumerate}
\item F-semistable at level $k$ if $Fut_k(\mathcal{X}, \mathcal{L})\ge 0$ holds for any test configuration $(\mathcal{X}, \mathcal{L})$ for $(X, -K_X)$ of exponent $k$. 
\item F-polystable at level $k$ if $X$ is F-semistable at level $k$ and $Fut_k(\mathcal{X}, \mathcal{L})=0$ if and only if $(\mathcal{X}, \mathcal{L})$ is a product configuration. 
\item F-stable at level $k$ if $X$ is F-semistable at level $k$ and $Fut_k(\mathcal{X}, \mathcal{L})=0$ if and only if $(\mathcal{X}, \mathcal{L})$ is trivial.  
\item asymptotically F-polystable $($resp. stable or semistable$)$ if there exists a $k_0>0$ such that $X$ is F-polystable $($resp. stable or semistable$)$ at level $k$ for all $k\ge k_0$. 
\end{enumerate}
\end{definition}

We conclude by pointing out that to test F-stability, we only need to consider normal test configurations.
\begin{lemma}
Let $(\mathcal{X}, \mathcal{L})$ be a test configuration for $(X, -K_X)$ of exponent $k$ and $(\widetilde{\mathcal{X}},\widetilde{\mathcal{L}})$ its normalization. Then, we have
\begin{align*}
Fut_k(\mathcal{X}, \mathcal{L})\ge Fut_k(\widetilde{\mathcal{X}},\widetilde{\mathcal{L}}). 
\end{align*}
\end{lemma}
\begin{proof}
This follows from \cite[Proposition 5.1]{RT07}, which says that
\begin{align}
DF(\mathcal{X}, \mathcal{L})\ge DF(\widetilde{\mathcal{X}},\widetilde{\mathcal{L}}), 
\quad Chow_k(\mathcal{X}, \mathcal{L})\ge Chow_k(\widetilde{\mathcal{X}},\widetilde{\mathcal{L}}). 
\quad\qedhere
\end{align}
\end{proof}

\section{Slope formula and F-stability of anti-canonically $k$-balanced metrics}
In this section we prove Theorem \ref{slope} and Theorem \ref{balanced}. In view of Proposition \ref{facts} (\ref{D^k}), we need the slope formulae of $D$ and $Z_k$. 
\begin{theorem}[{\cite[Proposition 3]{D05}}, or {\cite[Theorem 1]{P04}}+{\cite[Theorem 2.9]{M77}}]\label{Zslope}
Let $(\mathcal{X}, \mathcal{L})$ be a test configuration for $(X, -K_X)$ of exponent $k$, 
$H\in\mathcal{B}_k$ a Hermitian metric, and 
$(H_t)_t$ the Bergman geodesic ray associated with $(\mathcal{X}, \mathcal{L})$ and $H$. 
Then, we have
\begin{align*}
\lim_{t\to\infty} Z_k(H_t)
=\frac{(-K_X)^n}{n!}k^{n+1} Chow_k(\mathcal{X}, \mathcal{L}). 
\end{align*}
\end{theorem}

\begin{theorem}[{\cite[Theorem 3.11]{B16}}]\label{Dslope}
Let $(\mathcal{X}, \mathcal{L})$ be a normal test configuration for $(X, -K_X)$ of exponent $k$ and $\phi$ an $S^1$-invariant locally bounded metric on $(\mathcal{X}, \mathcal{L})\to\Delta$ with positive curvature current, where $\Delta\subset\mathbb{C}$ denotes the unit disc centerd at the origin. 
Then, setting $\phi_t:=\rho(\tau)^\ast\phi_\tau/k$, identified with a ray of metrics on $-K_X$ using the $\mathbb{C}^\ast$-action $\rho$ on $(\mathcal{X}, \mathcal{L})$, we have
\begin{align*}
DF(\mathcal{X}, \mathcal{L})=\lim_{t\to\infty}\frac{d}{dt} D(\phi_t)+q, 
\end{align*}
where $q$ is a non-negative rational number determined by the central fiber. The quantity $q$ vanishes if and only if $\mathcal{X}$ is $\mathbb{Q}$-Gorenstein with $\mathcal{L}$ isomorphic to $-kK_{\mathcal{X}/\mathbb{C}}$, and $\mathcal{X}_0$ is reduced, and its normalization has at worst log terminal singularities. 
\end{theorem}

\begin{proof}[Proof of Theorem \ref{slope}]
Let $(\mathcal{X}, \mathcal{L})$ be a normal test configuration for $(X, -K_X)$ of exponent $k$, $H\in\mathcal{B}_k$ a Hermitian metric and $(H_t)_t$ the Bergman geodesic ray associated with $(\mathcal{X}, \mathcal{L})$ and $H$. As proved in Lemma \ref{extension}, $(H_t)_t$ defines an $S^1$-invariant locally bounded metric $\phi$ on $(\mathcal{X}_\Delta, \mathcal{L}_\Delta)$ with positive curvature current. Note that 
\begin{align*}
\phi_t=\frac{1}{k}\rho(\tau)^\ast\phi_\tau=\rho(\tau)^\ast FS_k(H_\tau)
=FS_k(\rho(\tau)^\ast H_\tau)=FS_k(H_t). 
\end{align*}
Applying Theorem \ref{Zslope} and Theorem \ref{Dslope}
, we have
\begin{align*}
\lim_{t\to\infty}\frac{d}{dt} D^{(k)}(H_t)
&=\lim_{t\to\infty}\frac{d}{dt} D(FS(H_t))+\frac{n!}{k^{n+1}(-K_X)^n}\lim_{t\to\infty} Z_k(H_t)\\
&=\lim_{t\to\infty}\frac{d}{dt} D(\phi_t)+Chow_k(\mathcal{X}, \mathcal{L})\\
&=\frac{Fut_k(\mathcal{X}, \mathcal{L})}{kN_k}-q. \qedhere
\end{align*}
\end{proof}

\begin{proof}[Proof of Theorem \ref{balanced}]
Suppose that $X$ admits an anti-canonically $k$-balanced metric $H\in\mathcal{B}_k$. Let $(\mathcal{X}, \mathcal{L})$ be a normal test configuration for $(X, -K_X)$ of exponent $k$, and $(H_t)_t$ the Bergman geodesic ray associated with $H$ and $(\mathcal{X}, \mathcal{L})$. 
Since $H$ is a critical point of $D^{(k)}$, $D^{(k)}$ is convex along $(H_t)$ and $q$ is non-negative, we have
\begin{align*}
\frac{Fut_k(\mathcal{X}, \mathcal{L})}{kN_k}
=\lim_{t\to\infty}\frac{d}{dt}D^{(k)}(H_t)+q
\ge \lim_{t\to+0}\frac{d}{dt}D^{(k)}(H_t)
\ge 0. 
\end{align*}
This proves the F-semistability of $X$. 
Assume $Fut_k(\mathcal{X}, \mathcal{L})=0$. Since $q$ is nonnegative, $D^{(k)}(H_t)$ is affine in $t$. 
Then, Proposition \ref{facts} (\ref{product}) forces $(\mathcal{X}, \mathcal{L})$ to be a product configuration. 
\end{proof}

\begin{example}
Let $X_0$ be the Mukai-Umemura 3-fold, which is a compactification of the quotient of $SL(2, \mathbb{C})$ by the icosahedral group and $X$ a suitable small deformation of $X_0$. Both of them are Fano manifolds and $\mathfrak{h}(X_0)=\mathfrak{sl}(2, \mathbb{C})$ but $X$ does not admit non-trivial holomorphic vector fields, where $\mathfrak{h}(X_0)$ denotes the Lie algebra of all holomorphic vector fields on $X_0$. 
Tian constructed in \cite[Section 7]{T97} a special test configuration $(\mathcal{X}, \mathcal{L})$ for $(X, -K_X)$ of exponent 1 whose central fiber is $(X_0, -K_{X_0})$. 
Let $V$ be a holomorphic vector field on $X_0$ induced by the $\mathbb{C}^\ast$-action of $(\mathcal{X}, \mathcal{L})$. 
Fix a sufficiently large integer $m$, and consider a test configuration $(\mathcal{X}, m\mathcal{L})$. 
The expression $(\ref{Chow&hFut})$ shows that
\begin{align}\label{Fut&hFut}
Fut_m(\mathcal{X}, m\mathcal{L})=DF(\mathcal{X}, \mathcal{L})mN_m-\sum_{p=1}^n\frac{m^{n+1-p}}{(n+1-p)!}\mathcal{F}_{\mathrm{Td}^{(p)}}(V), 
\end{align}
where $\mathcal{F}_{\mathrm{Td}^{(p)}}$ denotes the $p$-th higher Futaki invariant on $X_0$. Since all the higher Futaki invariants are Lie algebra characters, $\mathfrak{h}(X_0)=\mathfrak{sl}(2, \mathbb{C})$ is semisimple, and $DF(\mathcal{X}, \mathcal{L})$ is a multiple of $\mathcal{F}_{\mathrm{Td}^{(1)}}(V)$, 
we have $Fut_m(\mathcal{X}, m\mathcal{L})=0$. Hence, $X$ is not asymptotically F-polystable and consequently does not admit any sequence of anti-canonically balanced metrics by Theorem \ref{balanced}. 
Although we will show in Proposition \ref{hFut} that higher Futaki invariants are obstructions to asymptotic F-polystability, they do not work well in this example because of the absence of non-trivial holomorphic vector fields. 
\end{example}

\section{F-stability and other stabilities}\label{vsF}
The aim of this section is to clarify the relation between asymptotic F-stability and other stabilities such as K-semistability, uniform K-stability, and asymptotic Chow stability. 

\begin{theorem}
Asymptotic F-semistability implies K-semistability. 
\end{theorem}
\begin{proof}
This is proved in the same line as Proposition \ref{Chow->K}. Since the Chow weight converges to 0 as raising the exponent, we have
\begin{align*}
\lim_{m\to\infty} \frac{Fut_{km}(\mathcal{X}, m\mathcal{L})}{kmN_{km}}=DF(\mathcal{X}, \mathcal{L}). \quad\qedhere
\end{align*}
\end{proof}

\begin{theorem}\label{MT->F}
Let $X$ be a Fano manifold. Suppose that the Ding functional of $X$ is $J$-coercive modulo $\mathrm{Aut}_0(X)$  and all the higher Futaki invariants of $X$ vanish. Then, $X$ is asymptotically F-polystable.  
\end{theorem}
Indeed, Berman-Witt Nystr\"om proved that under the same assumption, $X$ admits an anti-canonically $k$-balanced metric for sufficiently large $k$ (\cite[Theorem 1.7]{BW14}). 
Combining Theorem \ref{balanced}, we get the conclusion. 

In \cite[Theorem A]{BBJ15}, it was proved that a uniformly K-stable Fano manifold satisfies the assumpotion of Theorem \ref{MT->F}, and so we have the following:
\begin{corollary}
If a Fano manifold $(X, -K_X)$ is uniformly K-stable, then $X$ is asymptotically F-stable. 
\end{corollary}
For the definition of uniform K-stability, see \cite{BHJ15}. 
This is an analogue of \cite[Main Theorem]{MN15}, in which strong K-stability and asymptotic Chow stability are treated. 

We turn to asymptotic Chow stability. 
\begin{proof}[Proof of Therem \ref{Chow->F}]
This actually follows from the very definition of quantized Futaki invariants. 
Suppose that $(X, -K_X)$ is asymptotically Chow semistable. 
By Proposition ￼￼\ref{Chow->K}, this implies the K-semistability of $(X, -K_X)$. 
Then for any test configuration $(\mathcal{X}, \mathcal{L})$ with sufficiently large exponent $k$, we have
\begin{align*}
Fut_k(\mathcal{X}, \mathcal{L})=kN_k(DF(\mathcal{X}, \mathcal{L})+Chow_k(\mathcal{X}, \mathcal{L}))
\ge 0
\end{align*}
and the equality holds if and only if $DF(\mathcal{X}, \mathcal{L})=Chow_k(\mathcal{X}, \mathcal{L})=0$. This proves the asymptotic F-semistability of $X$. 
If we further assume that $(X, -K_X)$ is asymptotically Chow polystable (resp. stable), then $(\mathcal{X}, \mathcal{L})$ is a product (resp. trivial) configuration. This completes the proof. 
\end{proof}

The following proposition says that higher Futaki invariants also obstruct at least asymptotic F-polystability. 
\begin{proposition}\label{hFut}
If $X$ is asymptotically F-polystable, then all the higher Futaki invariants vanish on a maximal reductive subalgebra $\mathfrak{h}_r(X)$ of $\mathfrak{h}(X)$ 
\end{proposition}
\begin{proof}
Let $V$ be a holomorphic vector field $X$ whose real part generates $S^1$. 
Consider a product configuration $(\mathcal{X}, \mathcal{L})$ defined by $V$  with exponent $k$. 
The asymptotic F-polystability of $X$ forces $Fut_{km}(\mathcal{X}, m\mathcal{L})=0$ for sufficiently large $m$. Using the expression (\ref{Fut&hFut}), we get
\begin{align*}
DF(\mathcal{X}, \mathcal{L})kmN_{km}-\sum_{p=1}^n\frac{(km)^{n+1-p}}{(n+1-p)!}\mathcal{F}_{\mathrm{Td}^{(p)}}(V)=0, 
\end{align*}
which proves the proposition. 
\end{proof}

In the presence of K\"ahler-Einstein metircs, the converse of Theorem \ref{Chow->F} is also true. 
\begin{theorem}\label{Chow-F-hFut}
$X$ be a Fano manifold. Suppose that $X$ admits a K\"ahler-Einstein metric. Then, the following are equivalent:
\begin{enumerate}
\item $(X, -K_X)$ is asymptotically Chow polystable. 
\item $X$ is asymptotically F-polystable. 
\item All the higher Futaki invariants on $X$ vanish on a maximal reductive subalgebra $\mathfrak{h}_r(X)$ of $\mathfrak{h}(X)$.  
\end{enumerate}
\end{theorem}
\begin{proof}
The implication (a) $\Rightarrow$ (b) has been proved by Theorem \ref{Chow->F} and 
(b) $\Rightarrow$ (c) by Proposition \ref{hFut}. Note that these proofs do not use the existence of K\"ahler-Einstein metrics. 
(c) $\Rightarrow$ (a) is proved in \cite[Corollary 4.2]{F04}. 
\end{proof}

\begin{example}
In \cite{OSY12}, Ono-Sano-Yotsutani proved that there exists a toric Fano 7-manifold $X$ 
with K\"ahler-Einstein metrics, whose $p$-th higher Futaki invariant $\mathcal{F}_{\mathrm{Td}^{(p)}}$
 does not vanish for $p=2,\ldots, 7$. 
By Theorem \ref{Chow-F-hFut}, X is not asymptotically F-polystable and does not admit any sequence of anti-canonically balanced metrics by Theorem \ref{balanced}. 
\end{example}

\section{Lower bounds on the Calabi like functionals}
We devote this section to proving Theorem \ref{LB} as an application of Theorem \ref{slope}. 
Our approach is based on \cite{D05}. 

Let $X$ be an $n$-dimensional Fano manifold, and fix $k\ge 1$ so that $-kK_X$ is very ample. 

To begin with, we collect definitions. 
Let $H\in\mathcal{B}_k$ and $(s_\alpha)$ be an $H$-orthonormal basis for $H^0(X, -kK_X)$.  
We define a self-adjoint matrix $M(H)$ with entries
\begin{align*}
M(H)_{\alpha\beta}
&:=k^n\int_X\langle s_\alpha, s_\beta\rangle_{k FS_k(H)}\,\mu_{FS_k(H)}\\
&=k^n\int_X\frac{\langle s_\alpha, s_\beta\rangle_{k\phi}}{\sum |s_\gamma|^2_{k\phi}}\,\mu_{FS_k(H)}, 
\end{align*}
 where $\phi\in\mathcal{H}(X, -K_X)$ is any Hermitian metric on $-K_X$. 
Let $\underline{M}(H)$ be the trace-free part of $M(H)$, that is, 
\begin{align*}
\underline{M}(H)=M(H)-\frac{k^n}{N_k}\id. 
\end{align*}
This matrix appears in the derivative of quantized Ding functional: 
\begin{proposition}\label{derivative}
The derivative of $D^{(k)}$ along a Bergman geodesic ray $(H_t=e^{-tA}H)_t$ is given by
\begin{align*}
\frac{d}{dt}D^{(k)}(H_t)=\frac{1}{k^{n+1}}\tr(A\underline{M}(H_t)). 
\end{align*}
\end{proposition}
\begin{proof}
Define $s_\alpha^t:=e^{(t/2)A}s_\alpha$, so that $(s_\alpha^t)$ is an $H_t$-orthonormal basis. 
Let $(a_{\alpha\beta})$ denote the matrix representation of $A$ with respect to $(s_\alpha)$. Then, 
\begin{align*}
\frac{d}{dt}\mathcal{L}(FS_k(H_t))
&=\int_X\frac{d}{dt}FS_k(H_t)\,\mu_{FS_k(H_t)}\\
&=\frac{1}{k}\int_X \frac{\sum \langle a_{\alpha\beta}s_\beta^t, s_\alpha^t\rangle_{k\phi}}{\sum|s^t_\gamma|^2_{k\phi}}\, \mu_{FS_k(H_t)}\\
&=\frac{1}{k^{n+1}}\tr(AM(H_t)). 
\end{align*}
On the other hand, 
\begin{align*}
\frac{d}{dt}\mathcal{E}^{(k)}(H_t)=-\frac{1}{kN_k}\frac{d}{dt}\log\det e^{-tA}=\frac{\tr(A)}{kN_k}. 
\end{align*}
Combining them, we get the conclusion. 
\end{proof}
Note that this proves Proposition \ref{facts} (\ref{crit}). 
We recall the definition of $q$-norm of self-adjoint matrices for $q\ge 1$. 
For a self-adjoint matrix $A$, we define
\begin{align*}
||A||_q:=\left(\sum |\lambda_\alpha|^q\right)^{1/q}, 
\end{align*}
where $\lambda_\alpha$ denote the eigenvalues of $A$, repeated according to multiplicity. 

\begin{proposition}\label{UBonM}
For any $q>1$ and any Hermitian metric $\phi\in\mathcal{H}(X, -K_X)$, we have
\begin{align*}
||\underline{M}(Hilb_k(\phi))||_q
\le k^{n/q}||B(\phi)||_{L^q(\omega_\phi^n/n!)}+O(k^{(n/q)-1})
\end{align*}
\end{proposition}

Given a Hermitian metric $\phi\in\mathcal{H}(X, -K_X)$ on $-K_X$, we define the Bergman kernel to be
\begin{align*}
\rho_k(\omega_\phi):=\sum_{\alpha=1}^{N_k} |s_\alpha|^2_{k\phi}, 
\end{align*}
where $(s_\alpha)$ is a $Hilb_k(\phi)$-orthonormal basis for $H^0(X, -kK_X)$. We also use a scaled version of $\rho_k(\omega_\phi)$ defined by
\begin{align*}
\overline{\rho}_k(\omega_\phi):=\frac{1}{N_k}\rho_k(\omega_\phi). 
\end{align*}
One of the key ingredient in the proof of Proposition \ref{UBonM} is the asymptotic expansions of the Bergman kernels: 
\begin{theorem}[{\cite[Theorem 4.1.1]{MM}}]
We have the asymptotic expansions
\begin{align*}
\rho_k(\omega_\phi)&=(k^n+O(k^{n-1}))\frac{\omega_\phi^n}{n!\mu_\phi}, \\
\overline{\rho}_k(\omega_\phi)&=(1+O(k^{-1}))\frac{\omega_\phi^n}{n!\mu_\phi}, 
\end{align*}
valid in $C^l$ for any positive integer $l$. 
\end{theorem}

We define $T_k:=FS_k\circ Hilb_k$. Since $e^{-T_k(\phi)}=\overline{\rho}_k(\omega_\phi)^{-1/k}e^{-\phi}$, 
we have 
\begin{align*}
\int_X e^{-T_k(\phi)}=\int_Xe^{-\phi}+O(k^{-1})
\end{align*}
and
\begin{align*}
\mu_{T_k(\phi)}
=\frac{e^{-T_k(\phi)}}{\displaystyle\int_X e^{-T_k(\phi)}}
=\frac{\overline{\rho}_k(\omega_\phi)^{-1/k}}{\displaystyle\int_X e^{-\phi}+O(k^{-1})}\,e^{-\phi}
=(1+O(k^{-1}))\mu_\phi
\end{align*}
as $k\to\infty$. 

\begin{proof}[Proof of Proposition \ref{UBonM}]
Let $(s_\alpha)$ be an $Hilb_k(T_k(\phi))$-orhthogonal, $Hilb_k(\phi)$-orthonormal basis for $H^0(X, -kK_X)$. 
Then, $\underline{M}(Hilb_k(\phi))$ is a diagonal matrix with entries
\begin{align*}
\underline{M}(Hilb_k(\phi))_{\alpha\alpha}
&=k^n\int_X\frac{|s_\alpha|^2_{k\phi}}{\rho_k(\omega_\phi)}\, \mu_{T_k(\phi)}-\frac{k^n}{N_k}\int_X
|s_\alpha|^2_{k\phi}\,\mu_\phi\\
&=\int_X|s_\alpha|^2_{k\phi}\frac{k^n}{\rho_k(\omega_\phi)}(1+O(k^{-1}))\mu_\phi-\frac{k^n}{N_k}\int_X
|s_\alpha|^2_{k\phi}\,\mu_\phi\\
&=\int_X|s_\alpha|^2_{k\phi}\left(\frac{k^n}{\rho_k(\omega_\phi)}-\frac{k^n}{N_k}\right)\mu_\phi+O(k^{-1})\\
&=\int_X|s_\alpha|^2_{k\phi}\left\{\left(\frac{n!\mu_\phi}{\omega_\phi^n}-\frac{n!}{(-K_X)^n}\right)+O(k^{-1})\right\}\,\mu_\phi+O(k^{-1})\\
&=\int_X|s_\alpha|^2_{k\phi}B(\phi)\,\mu_\phi+O(k^{-1}), 
\end{align*}
where we have used the uniform boundedness of $k^n/\rho_k(\omega_\phi)$ in $k$. Let $\eta$ and $\nu$ be diagonal matrices with entries
\begin{align*}
\eta_{\alpha\alpha}&:=\int_X|s_\alpha|^2_{k\phi}B(\phi)\,\mu_\phi, \\
\nu_{\alpha\alpha}&:=\underline{M}(Hilb_k(\phi))_{\alpha\alpha}-\eta_{\alpha\alpha}=O(k^{-1}). 
\end{align*}
Write
\begin{align*}
|s_\alpha|^2_{k\phi}|B(\phi)|=|s_\alpha|^{2/p}_{k\phi}|s_\alpha|^{2/q}_{k\phi}|B(\phi)|, 
\end{align*}
where $p$ is the H\"older conjugate of $q$. Applying the Holder inequality we have 
\begin{align*}
|\eta_{\alpha\alpha}|\le \left(\int_X|s_\alpha|^2_{k\phi}\,\mu_\phi\right)^{1/p}\left(\int_X
|s_\alpha|^2_{k\phi}|B(\phi)|^q\,\mu_\phi
\right)^{1/q}. 
\end{align*}
Since $(s_\alpha)$ is $Hilb_k(\phi)$-orthonormal, this shows
\begin{align*}
||\eta||_q^q
=\sum_\alpha |\eta_{\alpha\alpha}|^q
\le\int_X\rho_k(\omega_\phi)|B(\phi)|^q\,\mu_\phi
\le k^n ||B(\phi)||^q_{L^q(\omega_\phi/n!)}+O(k^{n-1}). 
\end{align*}
On the other hand, since $\nu_{\alpha\alpha}=O(k^{-1})$, we get
\begin{align*}
||\nu||_q^q=N_k\cdot O(k^{-q})=O(k^{n-q}), 
\end{align*}
and so $||\nu||_q=O(k^{(n/q)-1})$. 
Consequently, we have
\begin{align*}
||\underline{M}(Hilb_k(\phi))||_q
\le ||\eta||_q+||\nu||_q
\le k^{n/q}||B(\phi)||_{L^q(\omega_\phi^n/n!)}+O(k^{(n/q)-1}). \quad\qedhere
\end{align*}
\end{proof}

\begin{proposition}\label{LBonM}
Let $p$ be the H\"older conjugate of $q$. 
Given a normal test configuration $(\mathcal{X}, \mathcal{L})$ of exponent $k$ and $H\in\mathcal{B}_k$, we have
\begin{align*}
||\underline{A}||_p\cdot ||\underline{M}(H)||_q\ge -k^{n+1}\frac{Fut_k(\mathcal{X}, \mathcal{L})}{kN_k}, 
\end{align*}
where $A$ denotes the infinitesimal generator of the $\mathbb{C}^\ast$-action on $H^0(X, -kK_X)$  corresponding to $(\mathcal{X}, \mathcal{L})$. 
\end{proposition}
\begin{proof}
Put $H_t:=e^{-tA}H$, so that $(H_t)$ is the Bergman geodesic ray associated with $(\mathcal{X}, \mathcal{L})$ and $H$. By Theorem \ref{slope} and Proposition \ref{facts} (\ref{convexity}), we get
\begin{align*}
\frac{Fut_k(\mathcal{X}, \mathcal{L})}{kN_k}
&\ge \frac{Fut_k(\mathcal{X}, \mathcal{L})}{kN_k}-q\\
&=\lim_{t\to\infty}\frac{d}{dt}D^{(k)}(H_t)\\
&\ge \lim_{t\to+0}\frac{d}{dt}D^{(k)}(H_t)\\
&=\frac{1}{k^{n+1}}\tr(A\underline{M}(H))\\
&=\frac{1}{k^{n+1}}\tr(\underline{A}\,\underline{M}(H))\\
&\ge -\frac{1}{k^{n+1}}||\underline{A}||_p ||\underline{M}(H)||_q, 
\end{align*}
where in the last line we have used the H\"older inequarity. 
\end{proof}

\begin{proof}[Proof of Theorem \ref{LB}]
Let $k$ be an exponent of $(\mathcal{X}, \mathcal{L})$, and set $H_{km}:=Hilb_{km}(\phi)\in\mathcal{B}_{km}$ for large $m$. 
Denote by $A_{km}$ the infinitesimal generator of the $\mathbb{C}^\ast$-action on $H^0(X, -kmK_X)$ corresponding to $(\mathcal{X}, m\mathcal{L})$. Applying Proposition \ref{LBonM} to them, we get
\begin{align*}
||\underline{A}_{km}||_p\cdot ||\underline{M}(H_{km})||_q
&\ge -(km)^{n+1}\frac{Fut_{km}(\mathcal{X}, m\mathcal{L})}{kmN_{km}}\\
&=-(km)^{n+1}(DF(\mathcal{X}, \mathcal{L})+O(m^{-1})). 
\end{align*}
By Proposition \ref{UBonM}, 
\begin{align*}
||\underline{M}(H_{km})||_q\le 
(km)^{n/q} ||B(\phi)||_{L^q(\omega_\phi^n)}+O(m^{(n/q)-1}). 
\end{align*}
Since $p$ is even, the definition of $p$-norm of test configurations gives
\begin{align*}
||\underline{A}_{km}||_p=\tr(\underline{A}_{km}^p)^{1/p}=||(\mathcal{X}, \mathcal{L})||_p(km)^{(n/p)+1}+O(m^{n/p}). 
\end{align*}
Putting the pieces above together, we have
\begin{align*}
||(\mathcal{X}, \mathcal{L})||_p\cdot ||B(\phi)||_{L^q(\omega_\phi^n/n!)}
\ge -DF(\mathcal{X}, \mathcal{L})+O(m^{-1}). 
\end{align*}
Taking a limit as $m\to\infty$ finishes the proof. 
\end{proof}

\bibliographystyle{amsalpha}

\begin{thebibliography}{A}

\bibitem{B16}
R.~J.~Berman, \emph{K-polystability of ${\mathbb Q}$-Fano varieties admitting K\"ahler-Einstein metrics}, Invent. math., \textbf{203} (2016), no. 3, 973--1025. 

\bibitem{BBGZ13}
R.~J.~Berman, S.~Boucksom, V.~Guedj and A.~Zeriahi, \emph{A variational approach to complex Monge-Amp\`ere equations}, Publ. Math. de l'IH\`ES, \textbf{117} (2013), 179--245.

\bibitem{BBJ15}
R.~J.~Berman, S.~Boucksom and M.~Jonsson, \emph{A variational approach to the Yau-Tian-Donaldson conjecture}, arXiv preprint, arXiv:1509.04561 (2015). 

\bibitem{BHJ15}
S.~Boucksom, T.~Hisamoto and M.~Jonsson, \emph{Uniform K-stability, Duistermaat-Heckman measures and singularities of pairs}, arXiv preprint, arXiv:1504.06568 (2015), To appear in Ann. Inst. Fourier. 

\bibitem{BW14}
R.~J.~Berman and D.~Witt Nystr\"om, \emph{Complex optimal transport and the pluripotential theory of K\"ahler-Ricci solitons}, arXiv preprint, arXiv:1401.8264 (2014).

\bibitem{D05}
S.~K.~Donaldson, \emph{Lower bounds on the Calabi functional},
J. Diff. Geom., \textbf{70} (2005), 453--472.

\bibitem{D09}
S.~K.~Donaldson, \emph{Some numerical results in complex differential geometry},
Pure Appl. Math., \textbf{5} (2009), 571--618.

\bibitem{DVZ12}
A.~Della Vedova and F.~Zuddas, \emph{Scalar curvature and asymptotic Chow stability of projective
bundles and blowups}, Trans. Amer. Math. Soc., \textbf{364} (2012), no. 12, 6495--6511.

\bibitem{F10}
J.~Fine, \emph{Calabi flow and projective embeddings}, 
J. Diff. Geom., \textbf{84} (2010), no.3, 489--523, with an appendix by K.~Liu, and X.~Ma. 

\bibitem{F04}
A.~Futaki, \emph{Asymptotic Chow semistability and integral invariants},
Int. J. Math., \textbf{15} (2004), 967--979. 

\bibitem{H16}
T.~Hisamoto, \emph{On the limit of spectral measures associated to a test configuration of a polarized K\"ahler manifold}, 
J. reine angew. Math., \textbf{713} (2016), 129–-148

\bibitem{MM}
X.~Ma and G.~Marinescu, \emph{Holomorphic Morse Inequalities and Bergman Kernels}, 
Progr. Math., vol. 254, Birkh\"auser Verlag, Basel, 2007. 

\bibitem{MN15}
T.~Mabuchi and Y.~Nitta, 
\emph{Strong K-stability and asymptotic Chow-stability}, 
in Geometry and Analysis on Manifolds, In Memory of Professor Shoshichi Kobayashi, 
(eds. T.~Ochiai et al.), Progress in Mathematics, \textbf{308} (2015), 405--411, Birkhauser.

\bibitem{M77}
D.~Mumford, \emph{Stability of projective varieties},
Enseignement Math., \textbf{23} (1977), 39--110.

\bibitem{OSY12}
H.~Ono, Y.~Sano, and N.~Yotsutani, \emph{An example of an asymptotically Chow unstable manifold with constant scalar curvature}, Annales de l'Institut Fourier, \textbf{62} (2012), no. 4, 1265--1287

\bibitem{P04}
S.~T.~Paul, \emph{Geometric analysis of Chow Mumford stability},
Adv. Math., \textbf{182} (2004), 333--356. 

\bibitem{PS03}
D.~H.~Phong and J.~Sturm, \emph{Stability, Energy Functionals, and K\"ahler-Einstein Metrics}, 
Commun. Anal. Geom., \textbf{11} (2003), no. 3, 565--597. 

\bibitem{PS07}
D.~H.~Phong and J.~Sturm, \emph{Test configurations for K-stability and geodesic rays},
J. Symplectic Geom., \textbf{5} (2007), 221--247. 

\bibitem{RT07}
J.~Ross and R.~P.~Thomas, \emph{A study of the Hilbert-Mumford criterion for the stability of projective varieties}, J. Algebraic Geom., \textbf{16} (2007), 201--255.

\bibitem{T97}
G.~Tian, \emph{K\"ahler-Einstein metrics with positive scalar curvature},
Invent. Math., \textbf{130} (1997), 1--39.

\bibitem{WN12}
D.~Witt Nystr\"om, \emph{Test configurations and Okounkov bodies}, 
Compositio Math., \textbf{148} (2012),  no. 6, 1736--1756.

\bibitem{WZZ16}
F.~Wang, B.~Zhou and X.~H.~Zhu, \emph{Modified Futaki invariant and equivariant Riemann-Roch formula}, Adv. Math., \textbf{289} (2016), 1205--1235. 
\end{thebibliography}

\end{document}